\documentclass[USenglish]{article}

\usepackage[utf8]{inputenc}
\usepackage[small]{dgruyter}
\usepackage{microtype}
\usepackage{bm}
\usepackage{graphicx}
\usepackage{subcaption}

\usepackage[utf8]{inputenc}

\usepackage{xcolor}

\usepackage{bbm}

\newcommand{\pr}{\mbox{\sf P}}

\newcommand{\tr}{\mbox{\sf Tr}}

\newcommand{\Real}{\mathbb R}

\newcommand{\bx}{{\bf x}}               

\newcommand{\be}{{\bf e}}
\newcommand{\by}{{\bf y}}               

\newcommand{\calt}{{\cal T}}

\newcommand{\g}{\lambda}                
\newcommand{\Lam}{\Lambda}               
   
\usepackage[linesnumbered,ruled]{algorithm2e}

\theoremstyle{dgthm}

\theoremstyle{dgdef}
\newtheorem{lem}{Lemma}

\newtheorem{remark}{Remark}%

\newtheorem{pro}{Proposition}

\newtheorem{rem}{Remark}

\newtheorem{asm}{Assumption}

\newcommand{\diag}{\mbox{\sf Diag}}
\newcommand{\ex}{{\bf\sf E}}     

\begin{document}

\author[1]{Soumyadip Ghosh}
\author[1]{Lior Horesh}
  \author[1]{Vasileios Kalantzis}
  \author[1]{Yingdong Lu}
  \author[1]{Tomasz Nowicki} 
  \author[1]{Shashanka Ubaru} 
  \runningauthor{Ghosh et al.}
  \affil[1]{IBM Research, 1101 Kitchawan Rd, Yorktown Heights, NY 10598, U.S.A.}
  \title{Counting Triangles of Graphs via Randomized Trace Estimation with Incomplete Matrix-Vector Products}
  \runningtitle{Counting Triangles of Graphs}
  \abstract{Counting triangles in graphs is a fundamental operation in network analysis, underpinning metrics such as clustering coefficients and serving as a signal for community detection, link prediction, and anomaly detection. The standard approach computes the trace of the cube of the adjacency matrix, but explicitly forming $\mathbf{A}^3$ is infeasible for large graphs. Hutchinson’s randomized trace estimator offers an efficient alternative by approximating $\operatorname{Tr}(\mathbf{A}^3)$ through stochastic averaging of quadratic forms, requiring only matrix–vector products with $\mathbf{A}$. However, in distributed and heterogeneous computing environments, observing all entries of these products can be costly due to communication overhead and straggler effects. To address this, we propose a new variant of Hutchinson’s estimator that operates under partial observation constraints, where both the number and identities of observed entries are random. We provide theoretical guarantees on unbiasedness, variance bounds, and sample complexity, and demonstrate through experiments on synthetic and real-world graphs that our method achieves accurate triangle count estimates while reducing synchronization costs. This work highlights the adaptability of randomized algorithms to modern computational architectures and opens avenues for efficient motif counting in large-scale network analytics.}
  
  \classification[AMS]{60-08; 65C05; 65F35}
  \journalname{Monte Carlo Methods and Applications}
  \startpage{1}




\maketitle

\section{Introduction}

Counting the number of triangles in graphs is an important computational kernel in 
network analysis~\cite{milo2002network}. Graph triangles represent the smallest nontrivial higher-order motif, and captures notions such as \emph{transitivity} 
and \emph{triadic closure} that underlie widely used statistics such as clustering coefficients. Moreover, high triangle density often correlates with strong community structure  and cohesive subgraph formation, and the triangle count 
serves as a signal for link prediction and recommendation in data collections 
represented through graphs. Applications of triangle counting  include anomaly 
and fraud detection, capturing social reinforcement and the formation of trust 
networks, identifying protein–protein interaction modules and functional motifs 
in metabolic and transcriptional networks, link prediction, recommendation, and 
anomaly detection, e.g., see~\cite{newman2003structure,tsourakakis2009spectral,gu2024compact,rougemont2003dna,yook2004functional,welser2007visualizing,kollias2024counting}. 

Let ${\cal G}$ denote an undirected, unweighted graph consisting of $N \in 
\mathbb{N}$ nodes. The number of triangles in the graph ${\cal G}$ is equal 
to $\Delta({\cal G})=\frac{1}{6}\operatorname{Tr}(\mathbf{A}^3)$, where 
$\operatorname{Tr}(.)$ denotes the trace operator of the $N\times N$ 
adjacency matrix $\mathbf{A}$. For all but small graphs ${\cal G}$, forming 
and storing the matrix power $\mathbf{A}^3$ explicitly becomes prohibitively 
expensive. In practice, most network analytic applications require only a 
crude approximation of $\Delta({\cal G})$, thus allowing the use of randomized 
trace estimators\footnote{See also \cite{kalantzis2013accelerating,bekas2007estimating} 
for randomized diagonal estimation and related applications.} as an alternative. 
In particular, Hutchinson's randomized trace estimator \cite{hutchinson1990stochastic} 
is known to achieve minimal variance among zero-mean, unit-variance distributions \cite{wu2016estimating}. The key observation underlying Hutchinson’s stochastic 
estimator is that for any zero-mean, random vector $\bm{x}\in \mathbb{R}^N$ with $\mathbb{E}[\bm{x}\bm{x}^{\top}] = \mathbf{I}$ (isotropic), the linearity of 
expectation and the cyclic property of the trace operator leads to 
$\mathbb{E}[\bm{x}^{\top} \mathbf{A} \bm{x}] = \operatorname{Tr}(\mathbf{A)}$.  
More generally, randomized trace estimators 
allow an unbiased approximation of the trace $\operatorname{Tr}(f(\mathbf{A}))$ of 
an implicitly-defined matrix function $f(\mathbf{A})$ via computing matrix-vector products with the matrix $\mathbf{A}$. In the case of triangle counting, $f(\mathbf{A})=\mathbf{A}^3$, and a randomized trace estimator for this problem 
can be found in \cite{Avron2010CountingTriangles}. Hutchinson's trace estimator 
reduces the problem of the trace computation to a stochastic average of quadratic 
forms, yielding an unbiased estimate with mean–square error that decays as 
$\mathcal{O}(1/n)$ and has standard deviation $\mathcal{O}(1/\sqrt{n})$. 
Further asymptotic improvements of Hutchinson's trace estimator are possible by either 
exploiting Fourier transformations \cite{avron2011randomized} or first capturing the 
dominant spectral subspace via a small randomized range finder and then applying Hutchinson’s estimator only to the orthogonal subspace  \cite{meyer2020hutchplusplus,persson2022improved}. 

The computational cost of Hutchinson's trace estimator is dominated by the computational 
cost required to compute matrix-vector products with the adjacency matrix $\mathbf{A}$. 
The cost for each such matrix-vector product runs at $O(\mathrm{nnz}(\mathbf{A}))$ where $\mathrm{nnz}(.)$ 
denotes the number of non-zero entries of the matrix $\mathbf{A}$. In this paper we consider 
the problem of triangle counting in computational systems where observing all $N$ entries of 
a matrix-vector product with the adjacency matrix $A$ can be (much) higher than 
$O(\mathrm{nnz}(\mathbf{A}))$. For example, consider that each matrix-vector multiplication with $\mathbf{A}$ 
is decomposed across a processor set $\mathcal{P} = \{1,\dots,P\}$, where each node owns a non-overlapping subset 
of rows of $\mathbf{A}$. As the cardinality 
of the set $\mathcal{P}$ increases, the per-processor computational cost associated with the computation of each local matrix-vector portion 
becomes smaller, while, at the same time, costs associated with data movement or communication increase. This 
phenomenon becomes more prevalent in distributed computing environments with non-dedicated resources 
where one or more processors can be much slower than its or their peers. Therefore, it is important to 
consider and analyze an asynchronous \cite{bertsekas2015parallel,tsitsiklis1986distributed,chow2015asynchronous} viewpoint of trace estimation. This topic has been already considered in 
\cite{kalantzis2024asynchronous} for symmetric matrix functions $f(\mathbf{A})$, however
in this work we focus explicitly in the case $f(\mathbf{A})=\mathbf{A}^3$. 

The contributions of this paper are as follows.
\begin{enumerate}
    \item We introduce a new variant of Hutchinson’s trace estimator to approximate $\operatorname{Tr}(\mathbf{A}^3)$ under the constraint that each matrix–vector product with the matrix $\mathbf{A}$ is only partially observed, and both the 
    number and identities of observed entries are random variables.
    \item We provide a theoretical analysis of the resulting estimator, including second-moment bounds and sample-complexity guarantees.
    \item We evaluate the proposed method on synthetic and real-world graphs, demonstrating its accuracy and practical effectiveness.
\end{enumerate}

Finally, we note that numerical algorithms that mitigate matrix-vector product constraints such as the one described above have been already suggested for 
other major numerical linear algebra kernels, e.g., see ~\cite{sabelfeld2022new} for a random columns sampling method, ~\cite{kalantzis2025straggler} 
for sparse linear system solvers, and \cite{ghosh2025power} for a power iteration 
variant to compute the dominant eigendirection of a symmetric adjacency matrix.

This paper is organized as follows. Section \ref{sec2} presents a new estimator to approximate 
$\Delta({\cal G})$ under scenarios where for any matrix-vector product the matrix $\mathbf{A}$ 
we observe only a random subset of the product entries. Section \ref{sec3} presents an analysis of the new 
estimator including second moment calculations and sample complexity. Section \ref{sec4} presents 
numerical experiments on a few test problems. Finally, Section \ref{sec5} presents our concluding 
remarks.

\subsection{Related work}

Algorithms for triangle counting of graphs are generally divided in two separate classes, 
exact counting algorithms and approximate counting algorithms. For an $N$-node graph 
${\cal G}$ with $M\in \mathbb{N}$ edges, a graph-based exact counting 
algorithm (\emph{edge-iterator}) traverses all nodes and for each pair of neighbors it checks 
whether they are connected by an edge, and thus form a closed triangle. The computational 
complexity of this approach runs at $O(N d_{max}^{2})$ for maximum degree $d_{max}\in 
\mathbb{N}$. For dense graphs, edge-iterator can be combined with randomized graph sparsification 
\cite{tsourakakis2009doulion}. A theoretical state-of-the-art algorithm, termed {\tt AYZ}, for exact triangle counting which runs in $O(M^{\frac{2 \; \omega}{\omega + 1}}) = O(M^{1.41})$ was advocated in \cite{alon1997finding}. High-performance implementations of triangle counting algorithms range from can be found in \cite{azad2015parallel, kolda2014counting, pearce2017triangle, Tom:2019:PTC:3337821.3337853,  wolf2017fast, zhang2019litete}. Triangle counting is also 
equivalent to computing the trace of matrix-functions by matrix multiplication via the 
formula $6\Delta({\cal G})=
\operatorname{Tr}(\mathbf{A}^3)$. Direct Matrix-Matrix multiplications results in complexity 
$O(N^{\omega})$ where $\omega$ is the matrix multiplication exponent; in 
\cite{coppersmith1990matrix} $\omega < 2.376$ and in \cite{doi:10.1137/1.9781611978322.63} $\omega = 2.371$.  

For large-scale graphs, the cost of exact counting can be impractical. This has motivated the development of \textit{approximate counting} techniques. The work in \cite{Avron2010CountingTriangles} employs randomized projections to estimate $\operatorname{Tr}(\mathbf{A}^3)$ while relying solely on matrix–vector products 
with the matrix $\mathbf{A}$. Similarly, eigenvalue–based approaches~\cite{4781156,tsourakakis2011counting} exploit the computation of the peripheral eigenvalues of $A$ in order to approximate triangle counts in ${\cal G}$, while the 
work in \cite{kollias2024counting} considers matrix-partition and low-rank submatrix approximations. 
These methods have also been combined with preprocessing steps for uniform edge sampling and sparsification~\cite{tsourakakis2009spectral}. Uniform edge sampling itself has been formalized as a general sparsification framework applicable to any triangle counting algorithm~\cite{tsourakakis2009doulion}, and later extended to incorporate vertex grouping strategies~\cite{10.1007/978-3-642-18009-5_3,alon1997finding}.

\subsection{Notations}

Lowercase bold letters, $\bm{a}$, denote vectors, and uppercase bold letters, $\mathbf{A}$, denote matrices. 
The $ij$-th entry of the matrix $\mathbf{A}$ is denoted by $A_{ij}$, while the $ij$-th entry of the 
matrix power $\mathbf{A}^k,\ k\in \mathbb{N}$, is denoted by $[\mathbf{A}^k]_{ij}$. Similarly, the $i$-th entry of a vector $\bm{x}$ will be denoted as $[\bm{x}]_i$. The norm $\|\cdot\|$ is refers to the Euclidean norm, i.e., 
$\|\bm{x}\|=(\sum_{i=1}^N [\bm{x}]_i^2)^\frac12$. The term $\mathbb{E}[\cdot]$ denotes the expectation operator. 
The symbol $\mathbb{N}$ denotes the set of positive integers.

\section{A new estimator with inexact matrix-vector products} \label{sec2}

In this section we present Hutchinson's trace estimator for triangle counting with partial 
(asynchronous) matrix-vector products. First, we briefly describe the classical 
approach. Let $m\in \mathbb{N}$ and consider $m$ independent realizations $\bm{x}_1,\ldots
\bm{x}_n$,  of the Rademacher vector 
where each entry of $\bm{x}_i$ is $\pm1$ independently with equal probability (Rademacher variables). Then, 
Hutchinson's trace estimator, $K_n$, provides an unbiased estimation of $\Delta({\cal G})$,
\begin{align*}
K_n=\frac{1}{n} \sum_{k=1}^n \bx_k^\top \mathbf{A}^3 \bx_k.
\end{align*}
Let now $T\in \{1,2,\ldots,N\}$ be an integer random variable, and define the random matrix 
\begin{equation*}
    \mathbf{I}_T=\sum_{i\in \calt} \be_i \be_i^\top,
\end{equation*}
for some random set $\calt \subseteq [N]$ of cardinality $T$, where $\be_i$ denotes the $i$-th column of the $N\times N$ identity matrix. Throughout the rest of this paper 
we will call $\mathbf{I}_T$ a \emph{row-selection matrix}. Furthermore, we make the following assumption on the random set $\calt$.
\begin{asm}
\label{asm:randomness_on_T}
For each $i, j\in [N]$, and $i\neq j$, $\pr[i \in \calt]=\pr[j \in \calt]$, and it is denoted as $p_1$. Furthermore, $\pr[i,j \in \calt]$ is independent of the selection of $i$ and $j$, and is denoted as $p_2$.
\end{asm}
Consider
\begin{align*}
L_n=\frac{N^3}{nT^3} \sum_{k=1}^n \bx_k^\top (\mathbf{I}_{T_{k,3}}\mathbf{A})(\mathbf{I}_{T_{k,2}}\mathbf{A})(\mathbf{I}_{T_{k,1}}\mathbf{A}\bx_k),
\end{align*}
where $T_{k,i}, i=1,2,3$ are three independent samples of $\mathbf{I}_T$.  
This formulation enables non-blocking progress across processing units, allowing late or straggling updates to contribute stochastically without enforcing global synchronization. 
When coupled with randomized estimators such as the Hutchinson trace estimator, asynchronous evaluation of \(\bx^{\top}\mathbf{A}^3\bx\) yields unbiased estimates of \(\mathrm{Tr}(\mathbf{A}^3)\) with controllable variance while significantly improving hardware utilization in bandwidth- and latency-limited environments. 
The resulting algorithmic framework integrates naturally with modern distributed linear algebra systems and serves as a scalable primitive for higher-order graph statistics, spectral moment estimation, and triangle counting in massive networks.

\subsection{Unbiasedness of the Estimation}
\label{sec:unbiasedmess}

The independence between the Rademacher vector and the row-selection matrices $I_{T_i}, i=1,2,3$ and among the row-selection matrices themselves leads to the easy conclusion that the estimator $L_n$ is unbiased. 
\begin{lem}
Under Assumption~\ref{asm:randomness_on_T}, $L_n$ is an unbiased estimator, i.e. $\ex[L_n]=\tr(\mathbf{A}^3)$.
\end{lem}
\begin{proof} Let $\ex_T$ denote the expectation operator taken with respect to the row selection matrices, then
\begin{align*}
\ex[L_n]=&\frac{N^3}{nT^3}\sum_{k=1}^n\ex[ \bx^\top (\mathbf{I}_{T_{k,3}}\mathbf{A})(\mathbf{I}_{T_{k,2}}\mathbf{A})(\mathbf{I}_{T_{k,1}}\mathbf{A}\bx)]\\=& \frac{N^3}{nT^3}\sum_{k=1}^n\ex[ \bx^\top \ex_T[(\mathbf{I}_{T_{k,3}}\mathbf{A})(\mathbf{I}_{T_{k,2}}\mathbf{A})(\mathbf{I}_{T_{k,1}}\mathbf{A})]\bx)]\\ &=\ex[ \bx^\top \mathbf{A}^3\bx]=\tr(\mathbf{A}^3).
\end{align*}
Thus confirms the desired unbiasedness. 
\end{proof}

\section{Analysis of the algorithm} 
\label{sec3}

In this section, we will provide probabilistic analysis on the estimator $L_n$, thus quantify the effectiveness of the Monte Carlo 
algorithm.

\subsection{Technical Lemmata of Randomized Linear Algebra}

First, we collect some technical lemmas that will be utilized heavily in 
the subsequent analysis.
\begin{lem}
\label{lem:4MRaemacher}
For any $N\times N$ matrix $\mathbf{Z}$ and $N$-dimensional Rademacher 
vector $\bx$, we have 
\begin{align}
\label{eqn:4MRaemacher}
\ex[(\bx\bx^\top)\mathbf{Z}(\bx\bx^\top)]=\mathbf{Z}+\mathbf{Z}^\top+\tr(\mathbf{Z}) \mathbf{I} -2\diag(\mathbf{Z}).
\end{align}
with $\diag (\mathbf{Z})$ denotes the diagonal matrix whose $(i,i)$-th entry is $Z_{ii}$, $i=1,2,\ldots, N$.
\end{lem}
\begin{proof} 
For any $i,j=1, \ldots, N$, we have, for $\mathbf{Z}=(Z_{ij})_{i,j=1}^N$,
\begin{align*}
\ex[(\bx\bx^\top)\mathbf{Z}(\bx\bx^\top)]_{ij}= \ex\left[\sum_{k, \ell=1}^N [\bx]_i[\bx]_k Z_{k\ell}[\bx]_\ell [\bx]_j\right].
\end{align*}
Therefore, in the diagonal case, that is $i=j$, we have, 
\begin{align*}
\ex[(\bx\bx^\top)\mathbf{Z}(\bx\bx^\top)]_{ii}= \ex\left[\sum_{k, \ell} [\bx]_i[\bx]_k Z_{k\ell}[\bx]_\ell [\bx]_i\right]=\ex\left[\sum_{k, \ell} [\bx]_k Z_{k\ell}[\bx]_\ell \right]=\tr(\mathbf{Z}).
\end{align*}
The second equality is due to the fact that $[\bx]_i^2=1$ always holds. 
In the case of $i\neq j$, 
\begin{align*}
\ex[(\bx\bx^\top)\mathbf{Z}(\bx\bx^\top)]_{ij}= &\ex\left[\sum_{k, \ell} [\bx]_i[\bx]_k Z_{k\ell}[\bx]_\ell [\bx]_j\right]\\=&\ex[[\bx]_i^2 Z_{ij }[\bx]_j^2 + [\bx]_i [\bx]_j Z_{j i} [\bx]_i [\bx]_j]= Z_{ij}+ Z_{ji}.
\end{align*}
It is easy to check that the above two cases match the right hand side of~\eqref {eqn:4MRaemacher}.
\end{proof}
\begin{remark}
There is a similar result for Gaussian sketching, where 
\begin{equation*}
\ex[(\by\by^\top)\mathbf{Z}(\by\by^\top)]=\mathbf{\Lam} (\mathbf{Z}+\mathbf{Z}^\top)\mathbf{\Lam} +\tr(\mathbf{Z}) \mathbf{\Lam},
\end{equation*}
with $\mathbf{\Lam}$ being the covariance matrix of the Gaussian vector $\by$, 
see, e.g.~\cite{horesh2025variance}. Similar results to the ones presented in 
this paper can also be developed if the Rademacher vectors are replaced by standard 
normal vectors. 
\end{remark}
\begin{lem} 
\label{lem:RandomTCT}
Under assumption~\ref{asm:randomness_on_T}, for any $N\times N$ matrix $\mathbf{Z}$, we have,
\begin{align*}
\ex[\mathbf{I}_T\mathbf{Z}\mathbf{I}_T]
= &\mathbf{Z}p_2+\diag (\mathbf{Z}) (p_1-p_2).
\end{align*}
\end{lem}
\begin{proof}
The lemma follows from the following direct calculation, 
\begin{align*}
\ex[\mathbf{I}_T\mathbf{Z}\mathbf{I}_T]= &\ex\left[\sum_{i\in \calt} \sum_{j\in\calt}\be_i \be_i^\top \mathbf{Z} \be_j \be_j^\top\right] =\ex\left[\sum_{i\in \calt} \sum_{j\in\calt}Z_{ij} \be_i  \be_j^\top\right] 
\\= &\mathbf{Z}p_2+\diag (\mathbf{Z}) (p_1-p_2).
\end{align*}
\end{proof}

\begin{lem}
\label{lem:2IT}
Under Assumption~\ref{asm:randomness_on_T}, for any $N\times N$ matrices $\mathbf{U}$, $\mathbf{V}$ and $\mathbf{W}$, we have 
\begin{equation*}
    \ex[\mathbf{I}_T \mathbf{U}\diag(\mathbf{V} \mathbf{I}_T \mathbf{W})]_{ij}= p_2 U_{ij} (\mathbf{VW})_{jj}-(p_2-p_1)U_{ij}V_{ji}W_{ij},
\end{equation*}
or, in matrix form, 
\begin{equation*}
\ex[\mathbf{I}_T \mathbf{U}\diag(\mathbf{V} \mathbf{I}_T \mathbf{W})]=p_2 \mathbf{U}\diag(\mathbf{VW})+(p_1-p_2)\mathbf{U} \odot \mathbf{V}^\top \odot \mathbf{W},
\end{equation*}
where $\odot$ denotes the Hadamard (element-wise) product.
\end{lem}
\begin{proof}
To compute $\ex[\mathbf{I}_T \mathbf{U}\diag(\mathbf{V} \mathbf{I}_T \mathbf{W})]$, for $i,j=1,2,\ldots, N$, we have, 
\begin{align*}
&\ex[\mathbf{I}_T \mathbf{U}\diag(\mathbf{V} \mathbf{I}_T \mathbf{W})]_{ij}= \ex[(\mathbf{I}_T)_{ii} U_{ij} \diag(\mathbf{V} \mathbf{I}_T \mathbf{V})_{j j}]
= \ex[(\mathbf{I}_T)_{ii} U_{ij} (\mathbf{V} \mathbf{I}_T \mathbf{W})_{j j}].
\end{align*}
Expand the last term, we can further write, 
\begin{align*}
\ex\left[\sum_k (\mathbf{I}_T)_{ii} U_{ij} V_{jk} (\mathbf{I}_T)_{kk}W_{k j}\right]
= &p_{2} \sum_k  U_{ij} V_{jk} W_{kj} - (p_2-p_1) U_{ij}V_{ji}W_{ij}
\\ = & p_2 U_{ij} (\mathbf{UV})_{jj}-(p_2-p_1)U_{ij}V_{ji}W_{ij}.
\end{align*}
\end{proof}

\subsection{Second Moment Calculation}

\subsubsection{A general Lemma}

\begin{lem}
\label{lem:AB_expression}
Under assumption~\ref{asm:randomness_on_T}, for any $N\times N$ matrix $\mathbf{B}$, we have,
\begin{align}
\ex[\tr(\bx^\top \mathbf{I}_T \mathbf{A B}\bx \bx^\top \mathbf{I}_T \mathbf{A B} \bx)]= &p_2\tr(\mathbf{ABAB})+p_2 \tr(\mathbf{AB})^2+ p_1\tr(\mathbf{ABB}^\top \mathbf{A}^\top)\nonumber \\ &-2p_2 \tr(\diag(\mathbf{AB}) \mathbf{AB}).
\label{eqn:AB_expression}
\end{align}
\end{lem}
\begin{proof}
First, write,
\begin{align*}
\ex[\tr(\bx^\top \mathbf{I}_T \mathbf{A B}\bx \bx^\top \mathbf{I}_T \mathbf{A B} \bx)]= \ex[\tr(\mathbf{I}_T \mathbf{A B}\bx \bx^\top \mathbf{I}_T \mathbf{A B} \bx\bx^\top)]
\end{align*}
By Lemma~\ref{lem:4MRaemacher}, we can take the expectation with respect to the Rademacher vector $\bx$, and obtain the following decomposition
\begin{align*}
&\underbrace{\ex[\tr(\mathbf{I}_T \mathbf{A B} \mathbf{I}_T \mathbf{A B} )]}_I+ \underbrace{\ex[\tr(\mathbf{I}_T \mathbf{A B B}^\top \mathbf{A}^\top \mathbf{I}_T  )]}_{II} + \underbrace{\ex^2[\tr(\mathbf{I}_T \mathbf{A B})]}_{III}-2 \underbrace{\ex[\tr(\mathbf{I}_T \mathbf{A B}\diag(\mathbf{I}_T\mathbf{AB}))]}_{IV}.
\end{align*}
Now, let us deal with each term separately. For the first term, we can write
\begin{align*}
I=\tr(\ex[\mathbf{I}_T \mathbf{A B} \mathbf{I}_T \mathbf{A B} ])=p_2\tr(\mathbf{ABAB})+ (p_1-p_2) \tr(\diag(\mathbf{AB}) \mathbf{AB}).
\end{align*}
Similarly, for the second term we have
\begin{align*}
II=&\tr(\ex[\mathbf{I}_T \mathbf{A B B}^\top \mathbf{A}^\top \mathbf{I}_T])\\
=&p_2\tr(\mathbf{AB}\mathbf{B}^\top \mathbf{A}^\top)+ (p_1-p_2) \tr(\diag(\mathbf{AB}\mathbf{B}^\top \mathbf{A}^\top))
\\=& p_2\tr(\mathbf{AB}\mathbf{B}^\top \mathbf{A}^\top)+ (p_1-p_2) \tr(\mathbf{AB}\mathbf{B}^\top \mathbf{A}^\top)\\
=&p_1\tr(\mathbf{AB}\mathbf{B}^\top \mathbf{A}^\top).
\end{align*}
Finally, the third and fourth terms can be expanded as
\begin{align*}
III= &\ex^2[\tr(\mathbf{I}_T \mathbf{A B})]\\ =&\ex\left[\sum_{i=1}^N \sum_{k=1}^N [I_T]_{ii}A_{ik}B_{ki}\sum_{j=1}^N \sum_{\ell=1}^N [I_T]_{jj}A_{j\ell}B_{\ell j} \right]
\\ =&
p_2\left[\sum_{i=1}^N \sum_{k=1}^N \sum_{j=1}^N \sum_{\ell=1}^N A_{ik}B_{ki}A_{j\ell}B_{\ell j} \right] +(p_1-p_2)\left[\sum_{i=1}^N \sum_{k=1}^N \sum_{\ell=1}^NA_{ik}B_{ki} A_{i\ell}B_{\ell i} \right]
\\ =&
p_2\left[\sum_{i=1}^N  \sum_{j=1}^N  [\mathbf{AB}]_{ii}[\mathbf{AB}]_{jj} \right] +(p_1-p_2)\left[\sum_{i=1}^N [\mathbf{AB}]_{ii}^2 \right]\\ =& p_2 \tr(\mathbf{AB})^2+(p_1-p_2) \tr(\diag(\mathbf{AB}) \mathbf{AB}),
\end{align*}
and
\begin{align*}
IV=&\tr(\ex[\mathbf{I}_T \mathbf{A B}\diag(\mathbf{I}_T\mathbf{AB})])\\
=&p_2 \tr(\diag(\mathbf{AB}) \mathbf{AB})+(p_1-p_2) \tr(\mathbf{AB}\odot \mathbf{AB})\\
=&p_1 \tr(\diag(\mathbf{AB}) \mathbf{AB}),
\end{align*}
respectively. Combining all of the above terms together, results to ~\eqref{eqn:AB_expression}.
\end{proof}
\begin{rem}
We can also have the following more concise form of~\eqref{eqn:AB_expression}, where 
for some $N\times N$ matrix $\mathbf{C}$, we can write 
\begin{align}
\label{C:expression}
\ex[\tr(\bx^\top \mathbf{I}_T \mathbf{C}\bx \bx^\top \mathbf{I}_T \mathbf{C} \bx)]=p_2\tr(\mathbf{C}^2)+p_2 \tr(\mathbf{C})^2+ p_1\tr(\mathbf{CC}^\top )-2p_2 \tr(\diag(\mathbf{C}) \mathbf{C}).
\end{align}
\end{rem}
\subsubsection{Calculating $\ex[L_n^2]$}

Now we apply Lemma~\ref{lem:AB_expression} with $\mathbf{B}=\mathbf{I}_{T_2}\mathbf{A}\mathbf{I}_{T_3}\mathbf{A}$ to derive 
the second moment calculation of the the power trace. The four terms in the 
expression are calculated one by one as follows. 

\noindent
{\bf The first term:}
\begin{align*}
&\ex[\tr(\mathbf{AB}\mathbf{AB})]\\=&\ex[\tr(\mathbf{AI}_{T_2}\mathbf{AI}_{T_3}\mathbf{AAI}_{T_2}\mathbf{AI}_{T_3}\mathbf{A})]\\
=&\ex[\tr(\mathbf{AI}_{T_2}\mathbf{AI}_{T_3}\mathbf{A}^2\mathbf{I}_{T_2}\mathbf{AI}_{T_3}\mathbf{A})]
\\
=&\ex[\tr(\mathbf{AI}_{T_2}\mathbf{A}\ex_3[\mathbf{I}_{T_3}\mathbf{A}^2\mathbf{I}_{T_2}\mathbf{AI}_{T_3}]\mathbf{A})]
\\
\stackrel{(1)}{=}&p_2\ex[\tr(\mathbf{AI}_{T_2}\mathbf{A}^3\mathbf{I}_{T_2}\mathbf{A}^2)] + (p_1-p_2)\ex[\tr(\mathbf{AI}_{T_2}\mathbf{A}\diag(\mathbf{A}^2\mathbf{I}_{T_2}\mathbf{A})\mathbf{A})]
\\
\stackrel{(2)}{=}& p_2^2\tr(\mathbf{A}^6) + p_2(p_1-p_2)\tr(\mathbf{A}^3\diag (\mathbf{A}^3))+p_2(p_1-p_2) \tr(\mathbf{A}^3\diag(\mathbf{A}^3)) 
\\ &+ (p_1-p_2)^2\tr(\mathbf{A}(\mathbf{A}\odot \mathbf{A}^2 \odot \mathbf{A})\mathbf{A})
\\
=& p_2^2\tr(\mathbf{A}^6) + 2p_2(p_1-p_2)\tr(\mathbf{A}^3\diag (\mathbf{A}^3)) + (p_1-p_2)^2\tr((\mathbf{A\odot A}^2)^2),
\end{align*}
where (1) is the result of applying Lemma~\ref{lem:RandomTCT} and (2) is the result of applying Lemma~\ref{lem:RandomTCT} to the first term and Lemma~\ref{lem:2IT} to the second term. To see that last equality, notice that since $A$ is symmetric,
\begin{align*}
\tr(\mathbf{A}(\mathbf{A}\odot \mathbf{A}^2 \odot \mathbf{A})\mathbf{A})=&\sum_{i,j,k}A_{ij}[\mathbf{A}\odot \mathbf{A}^2 \odot \mathbf{A}]_{jk}A_{ki}\\
=& \sum_{j,k} [\mathbf{A}^2]_{jk}[\mathbf{A}\odot \mathbf{A}^2 \odot \mathbf{A}]_{jk}\\
=&\sum_{j,k}[\mathbf{A\odot A}^2]_{jk}[\mathbf{A\odot A}^2]_{kj}.
\end{align*}

\noindent
{\bf The second term:}
\begin{align*}
&\ex[\tr^2(\mathbf{AB})]
=\ex[\tr(\mathbf{AI}_{T_2}\mathbf{AI}_{T_3}\mathbf{A})\tr(\mathbf{AI}_{T_2}\mathbf{AI}_{T_3}\mathbf{A})]
\\=&\ex\left[\sum_{i,j,k}A_{ij}[I_{T_2}]_{jj}A_{jk}[I_{T_3}]_{kk} A_{ki}\sum_{i',j',k'}A_{i'j'}[I_{T_2}]_{j'j'}A_{j'k'}[I_{T_3}]_{k'k'} A_{k'i'}\right]
\\= &p_2^2\sum_{i,j,k}A_{ij}A_{jk}A_{ki}\sum_{i',j',k'}A_{i'j'}A_{j'k'} A_{k'i'}
\\&+p_2(p_1-p_2)\sum_{i,j,k}A_{ij}A_{jk}A_{ki}\sum_{i',k'}A_{i'j}A_{jk'}A_{k'i'}
\\&+p_2(p_1-p_2)\sum_{i,j,k}A_{ij}A_{jk}A_{ki}\sum_{i',j'}A_{i'j'}A_{j'k}A_{ki'}
\\&+(p_1-p_2)^2\sum_{i,j,k}A_{ij}A_{jk}A_{ki}\sum_{i'}A_{i'j}A_{jk}A_{ki'}
\\= &p_2^2\tr^2(\mathbf{A}^3)+ 2p_2(p_1-p_2)\tr(\mathbf{A}^3\diag (\mathbf{A}^3))+(p_1-p_2)^2\tr((\mathbf{A}\odot \mathbf{A}^2)^2).
\end{align*}
To see why the last equation holds, notice that 
\begin{align*}
&\sum_{i,j,k,i',k'}A_{ij}A_{jk}A_{ki}A_{i'j}A_{jk'}A_{k'i'}=\sum_{i,j,i'}A_{ij}[A^2]_{ji}A_{i'j}[A^2]_{ji'}=\sum_{j}[A^3]_{jj}[A^3]_{jj},
\end{align*}
\begin{align*}
\sum_{i,j,k, i'}A_{ij}A_{jk}A_{ki}A_{i'j}A_{jk}A_{ki'}=\sum_{j,k}[A^2]_{jk}A_{jk}[A^2]_{jk}A_{jk} = \sum_{j,k} (\mathbf{A}\odot \mathbf{A}^2)_{jk}(\mathbf{A}\odot \mathbf{A}^2)_{kj}.
\end{align*}

\noindent
{\bf The third term:}
\begin{align*}
& \ex[\tr(\mathbf{AB}\mathbf{B}^\top \mathbf{A}^\top)]
\\=&\ex[\tr(\mathbf{A}\mathbf{I}_{T_2}\mathbf{A}\mathbf{I}_{T_3}\mathbf{A}\mathbf{A}\mathbf{I}_{T_3}\mathbf{A}\mathbf{I}_{T_2}\mathbf{A}])
\\
=& \ex[\tr(\mathbf{A}I_{T_2}\mathbf{A}\ex_3[\mathbf{I}_{T_3}\mathbf{A}^2\mathbf{I}_{T_3}]\mathbf{A}\mathbf{I}_{T_2}\mathbf{A}])
\\
\stackrel{(1)}{=}& p_2\ex[\tr(\mathbf{A}\mathbf{I}_{T_2}\mathbf{A}^4\mathbf{I}_{T_2}\mathbf{A}])+(p_1-p_2)\ex[\tr(\mathbf{A}\mathbf{I}_{T_2}\mathbf{A}\diag(\mathbf{A}^2)\mathbf{A}\mathbf{I}_{T_2}\mathbf{A}])
\\
\stackrel{(2)}{=}& p^2_2\tr(\mathbf{A}^6)+p_2(p_1-p_2)\tr(\diag(\mathbf{A}^4)\mathbf{A}^2)+p_2(p_1-p_2)\tr(\diag(\mathbf{A}^2)\mathbf{A}^4])
\\ &+ (p_1-p_2)^2\tr(\diag (\mathbf{A}^4\diag(\mathbf{A}^2))
\\
=& p^2_2\tr(\mathbf{A}^6)+(p^2_1-p^2_2)\tr(\diag(\mathbf{A}^4)\mathbf{A}^2),
\end{align*}
where, once again, (1) and (2) result by applying Lemma~\ref{lem:RandomTCT}.

\vskip 0.3cm
\noindent
{\bf The fourth term:}
\begin{align*}
&\ex[\tr(\diag(\mathbf{AB})\mathbf{AB})]
=\ex[\tr(\mathbf{AI}_{T_2}A\ex_3[\mathbf{I}_{T_3}A\diag(\mathbf{AI}_{T_2}\mathbf{AI}_{T_3}\mathbf{A})])]
\\
=&p_2\ex[\tr(\mathbf{AI}_{T_2}\mathbf{A}^2\diag(\mathbf{AI}_{T_2}\mathbf{A}^2))]+ (p_1-p_2) \ex[\tr(\mathbf{AI}_{T_2}\mathbf{A}(\mathbf{A}\odot (\mathbf{AI}_{T_2}\mathbf{A})^\top \odot \mathbf{A})])]
\\
=&p_2^2 \tr(\mathbf{A}^3\diag(\mathbf{A}^3)) + p_2(p_1-p_2) \tr(\mathbf{A}(\mathbf{A}^2\odot \mathbf{A}\odot \mathbf{A}^2))
\\ &+(p_1-p_2)\ex\left[\sum_{i,j,k} A_{ij}[\mathbf{I}_{T_2}]_{jj}A_{jk} [\mathbf{A}\odot (\mathbf{A}\mathbf{I}_{T_2}\mathbf{A})^\top \odot \mathbf{A}]_{ki}\right]
\\
=& p_2^2 \tr(\mathbf{A}^3\diag(\mathbf{A}^3)) + 2p_2(p_1-p_2) \tr((\mathbf{A}\odot \mathbf{A}^2)^2)+(p_1-p_2)^2\tr((\mathbf{A}\odot \mathbf{A})^3).
\end{align*}
To see the above, notice that 
\begin{align*}
\tr(\mathbf{A}(\mathbf{A}^2\odot \mathbf{A}\odot \mathbf{A}^2))&=\sum_{i,j} A_{ij} [\mathbf{A}^2\odot \mathbf{A}\odot \mathbf{A}^2]_{ji}\\&=\sum_{i,j} A_{ij} [\mathbf{A}^2]_{ij} A_{ij} [\mathbf{A}^2]_{ii}\\&= \tr((\mathbf{A}\odot \mathbf{A}^2)^2),
\end{align*}
while, at the same time,
\begin{align*}
&\ex\left[\sum_{i,j,k} A_{ij}[\mathbf{I}_{T_2}]_{jj}A_{jk} [\mathbf{A}\odot (\mathbf{AI}_{T_2}\mathbf{A})^\top \odot \mathbf{A}]_{ki}\right]
\\=&
\ex\left[\sum_{i,j,k} A_{ij}[\mathbf{I}_{T_2}]_{jj}A_{jk} A_{ki}[\mathbf{A}\mathbf{I}_{T_2}\mathbf{A}]_{ik}  A_{ki}\right]
\\=&
\ex\left[\sum_{i,j,k, \ell} A_{ij}[\mathbf{I}_{T_2}]_{jj}A_{jk} A_{ki} A_{i\ell}[\mathbf{I}_{T_2}]_{\ell\ell}A_{\ell k}  A_{ki}\right]
\\=&
p_2\sum_{i,j,k, \ell} A_{ij}A_{jk} A_{ki} A_{i\ell}A_{\ell k}A_{ki}
+
(p_1-p_2)\sum_{i,j,k} A_{ij}A_{jk} A_{ki} A_{ij}A_{j k}A_{ki}
\\=&
p_2\sum_{i,k} [\mathbf{A}^2]_{ik} A_{ki} [\mathbf{A}^2]_{i k} A_{ki}
+
(p_1-p_2)\sum_{i,j,k} A_{ij}A_{jk} A_{ki} A_{ij}A_{j k}A_{ki}
\\=&
p_2\tr((\mathbf{A}\odot \mathbf{A}^2)^2)+ (p_1-p_2)\tr((\mathbf{A}\odot \mathbf{A})^3).
\end{align*}
Put these four terms together, we obtain the following expression for the second moment, 
\begin{pro}
\begin{align}
\ex[L_n^2]=&(p_1+p_2) p_2^2\tr(\mathbf{A}^6) + p_1(p_1^2-p_2^2) \tr(\mathbf{A}^2\diag (\mathbf{A}^4)) \nonumber \\ &+[4p_2^2(p_1-p_2) -2p_2^3]\tr(\mathbf{A}^3\diag (\mathbf{A}^3))
\nonumber \\
&+p^3_2\tr^2 (\mathbf{A}^3)+2p_2(p_1-p_2)(p_1-3p_2)\tr((\mathbf{A}\odot \mathbf{A}^2)^2)  \nonumber \\ &-2p_2(p_1-p_2)^2 \tr((\mathbf{A}\odot \mathbf{A})^3).\label{eqn:2nd_moment}
\end{align}
\end{pro}
\begin{rem}
Recall that $p_2$ is roughly $p_1^2$, so the dominant term in~\eqref{eqn:2nd_moment} is the second one, which is of the order of $p_1^3$. For the variance calculation, the square of the average needs to be deducted, but that terms will be of $p_1^6$ order, not making much difference. 
\end{rem}

\subsection{Sample Complexity}

Next, we develop sample complexity results similar to that in~\cite{kalantzis2024asynchronous} with the techniques on concentration inequalities developed in~\cite{zhou2019sparse}.

\begin{pro}
\label{pro:sample_complexity}
Under the Assumption~\ref{asm:randomness_on_T}, there exists a $n_0$, such that when $n\ge n_0$, the asynchronous estimator $L_n$ is an $(\epsilon, \delta)$-estimator of $\tr(\mathbf{A}^3)$, i.e.
\begin{align}
\label{eqn:sample_complexity}
\pr\left[|L_n - \tr(\mathbf{A}^3)|\le \epsilon \right] \ge 1-\delta,
\end{align}
with $n_0=Cp_1^6\epsilon^{-2} \log(\delta^{-1})$ for a $C>0$.
\end{pro}
\begin{proof}
Lemma 3.2 in~\cite{zhou2019sparse} states that whenever a mean-zero sub-Gaussian random variable $X$ satisfies\\ $\sup\limits_{p\ge 1}\left(p^{-\frac12}(\ex|X|^p)^{1/p}\right)<K$ for certain $K>0$, then for $A>0$ any $a \in \Real$ such that $|a| \le A$ and $\g<1/(4eK^2A)$, we have
\begin{align}
\label{eqn:lemma32}
\ex(\exp(\g a X^2) \le 1+ \g a \ex[X^2] +16 \g^2 a^2 K^4.
\end{align}
This lemma allows us to utilize the second moment calculation in the large deviations calculations. More specifically, for any $t>0$, the probability that the estimator will have an error more than $t$
\begin{align*}
\pr\left[\bigg|
\sum_{i=1}^N \sum_{j,k=1}^N (\mathbf{I}_{T_1})_{ii} A_{ij} [\mathbf{I}_{T_2}]_{jj} A_{jk} [\mathbf{I}_{T_3}]_{kk} A_{ki} [\mathbf{x}]_i^2 -p_1^3\sum_{i=1}^N[\mathbf{A}^3]_{ii}\bigg|>t\right],
\end{align*}
can be bounded via the Chebyshev's inequality by 
\begin{align*}
&\ex\left[\exp\left(\g\sum_{i=1}^N \sum_{j,k=1}^N [\mathbf{I}_{T_1}]_{ii} A_{ij} [\mathbf{I}_{T_2}]_{jj} A_{jk} [\mathbf{I}_{T_3}]_{kk} A_{ki} [\mathbf{x}]_i^2 -\g p_1^3\sum_{i=1}^N[\mathbf{A}^3]_{ii}\right)\right]e^{-\g t}
\\
=&\ex\left[\exp\left(\g\sum_{i=1}^N \sum_{j,k=1}^N [\mathbf{I}_{T_1}]_{ii} A_{ij} [\mathbf{I}_{T_2}]_{jj} A_{jk} [\mathbf{I}_{T_3}]_{kk} A_{ki} [\mathbf{x}]_i^2 \right)\right]
\left(\prod_{i=1}^N\exp(\g p_1^3[\mathbf{A}^3]_{ii})\right)^{-1}e^{-\g t}
\end{align*}
for any $\g>0$. Thus, we would like to calculate the following moment generating function, 
\begin{align*}
&\ex\left[\exp\left(\g\sum_{i=1}^N \sum_{j,k=1}^N [\mathbf{I}_{T_1}]_{ii} A_{ij} [\mathbf{I}_{T_2}]_{jj} A_{jk} [\mathbf{I}_{T_3}]_{kk} A_{ki} [\mathbf{x}]_i^2 \right)\right]
\\=&
\ex_T\left[\ex_x\left[\exp\left(\g\sum_{i=1}^N \sum_{j,k=1}^N [\mathbf{I}_{T_1}]_{ii} A_{ij} [\mathbf{I}_{T_2}]_{jj} A_{jk} [\mathbf{A}_{T_3}]_{kk} A_{ki} [\mathbf{x}]_i^2 \right)\right]\right]
\end{align*}
Denote
\begin{align*}
\Lam_i:=\sum_{j,k=1}^N [\mathbf{I}_{T_1}]_{ii} A_{ij} [\mathbf{I}_{T_2}]_{jj} A_{jk} [\mathbf{I}_{T_3}]_{kk} A_{ki}, 
\end{align*}
Therefore, $\ex[\Lam_i]=p_1^3 A_{ii}$. Furthermore, when $\g<1/(4eN^3)$, we have, 
\begin{align*}
\ex_T\left[\ex_x\left[\exp\left(\g\sum_{i=1}^N \Lam_i [\mathbf{x}]_i^2\right)\right]\right]&=\ex_T\left[\prod_{i=1}^N\ex_x\left[\exp\left(\g \Lam_i [\mathbf{x}]_i^2\right)\right]\right]
\\
&\le \ex_T\left[\prod_{i=1}^N(1+\g \Lam_i + 16\g^2 \Lam^2_i )\right]
\end{align*}
where the inequality follows from \eqref{eqn:lemma32} with $K=1$ due to the Rademacher assumption. Because the entries of the incidence matrix $\mathbf{A}$ are binary, we also 
have 
\begin{align}
\label{eqn:dropingout} \Lam_i \le \sum_{j,k=1}^N [\mathbf{I}_{T_1}]_{ii} A_{ij}  A_{jk}  A_{ki}= [\mathbf{I}_{T_1}]_{ii} [\mathbf{A}^3]_{ii}. 
\end{align}
Hence,
\begin{align*}
\ex_T\left[\prod_{i=1}^N(1+\g \Lam_i + 16\g^2 \Lam^2_iK^4 )\right]=& 1+ \g \sum_{i=1}^N \ex[\Lam_i] + 16\g^2  \sum_{i=1}^N \ex[\Lam^2_i]\\ &+ \sum_{k, m}(16)^m\g^{k+2m}\ex[\Lam_{i_1}\cdots  \Lam_{i_k}\Lam_{j_1}\cdots  \Lam_{j_m}]. 
\end{align*}
For each term in the last summation, we apply the inequality~\eqref{eqn:dropingout}, and obtain, 
\begin{align*}
\ex[\Lam_{i_1}\cdots  \Lam_{i_k}[\mathbf{\Lam}^2]_{j_1}\cdots  [\mathbf{\Lam}^2]_{j_m}] \le p_1^{k+m} [\mathbf{A}^3]_{i_1, i_1}\cdots [\mathbf{A}^3]_{i_k, i_k} [\mathbf{A}^6]_{j_1, j_1} \cdots [\mathbf{A}^6]_{j_1, j_1}.
\end{align*}
Therefore, 
\begin{align*}
&\sum_{k, m}(16K^4)^m\g^{k+2m}\ex[\Lam_{i_1}\cdots  \Lam_{i_k}\Lam_{j_1}\cdots  \Lam_{j_m}] \\
\le & \prod_{i=1}^N(1+p_1 \g \mathbf{A}^3_{ii} + 16p_1 \g^2 [\mathbf{A}^6]_{ii} )-\left(1+ \g \sum_{i=1}^N \ex[\Lam_i]\right)=:\Omega(\g).
\end{align*}
Note that each term in $\Omega(\g)$ contains $\g^p$ with $p\ge 2$, and there exists an $\Omega^*>0$ such that $\Omega(\g)\le \g^2 \Omega^*$ for all $\g<1/(4eN^3)$. This leads to
\begin{align*}
\ex_T\left[\ex_x\left[\exp\left(\g\sum_{i=1}^N \Lam_i x_i^2\right)\right]\right] \le & \ex_T\left[\prod_{i=1}^N\{1+\g \Lam_i + 16\g^2 [(I_{T_1})_{ii} [\mathbf{A}^3]_{ii}]^2 \}\right]
\\ \le &1+ \g \sum_{i=1}^N \ex[\Lam_i] + 16\g^2  \sum_{i=1}^N \ex[\Lam^2_i] + \Omega(\g).
\end{align*}
From the basic inequality $1+x\le e^x$, we have, 
\begin{align*}
1+ \g \sum_{i=1}^N \ex[\Lam_i] + 16\g^2  \sum_{i=1}^N \ex[\Lam^2_i] + \Omega
\le & \exp \left[  \g \sum_{i=1}^N \ex[\Lam_i] + 16\g^2  \sum_{i=1}^N \ex[\Lam^2_i] + \Omega\right]
\\=& \prod_{i=1}^N \exp \left(\g\ex[\Lam_i]+16\g^2 \ex[\Lam^2_i] +\frac{\Omega(\g)}{N}\right).
\end{align*}
Hence,
\begin{align*}
&\ex\left[\exp\left(\g\sum_{i=1}^N \sum_{j,k=1}^N [\mathbf{I}_{T_1}]_{ii} A_{ij} [\mathbf{I}_{T_2}]_{jj} A_{jk} [\mathbf{A}_{T_3}]_{kk} A_{ki} [\mathbf{x}]_i^2 -\g p_1^3\sum_{i=1}^N[\mathbf{A}^3]_{ii}\right)\right]
\\
\le & \left( \prod_{i=1}^N \exp \left(\g\ex[\Lam_i]+16\g^2 \ex[\Lam^2_i] +\frac{\Omega}{N}\right)\right)\left(\prod_{i=1}^N\exp(\g p_1^3[\mathbf{A}^3]_{ii})\right)^{-1}
\\ = & \exp \left(16\g^2 \sum_{i=1}^N\ex[\Lam^2_i] +\Omega(\g)\right).
\end{align*}
Hence, 
\begin{align*}
&\pr\left[\bigg|
\sum_{i=1}^N \sum_{j,k=1}^N [\mathbf{I}_{T_1}]_{ii} A_{ij} [I_{T_2}]_{jj} A_{jk} [\mathbf{I}_{T_3}]_{kk} A_{ki} [\mathbf{x}]_i^2 -p_1^3\sum_{i=1}^N[\mathbf{A}^3]_{ii}\bigg|>t\right]\\ \le& \exp\left(-\g t + 16\g^2\sum_{i=1}^N\ex[\Lam^2_i] +\Omega(\g)\right).
\end{align*}
$\g$ can then be optimized to get a better bound. But for the ease of exposition, we take $\g^*= \arg\min_t\{ -\g t + 16\g^2\sum_{i=1}^N\ex[\Lam^2_i] \g^2\Omega^* \}=\frac{t}{32\sum_{i=1}^N\ex[\Lam^2_i]+2\Omega^*}$, and 
\begin{align}
&\pr\left[\bigg|
\sum_{i=1}^N \sum_{j,k=1}^N [\mathbf{I}_{T_1}]_{ii} A_{ij} [\mathbf{I}_{T_2}]_{jj} A_{jk} [\mathbf{I}_{T_3}]_{kk} A_{ki} [\mathbf{x}]_i^2 -p_1^3\sum_{i=1}^N[\mathbf{A}^3]_{ii}\bigg|>t\right]
\nonumber
\\
\le & \exp\left(-\frac{t^2}{ 64\sum_{i=1}^N\ex[\Lam^2_i]+4\Omega^*} \right).\label{eqn:diag_inq}
\end{align}
For any $i=1,2, \ldots, N$, 
\begin{align*}
\ex[\Lam^2_i]=&\ex\left[\sum_{j,k, j',k'=1}^N [\mathbf{I}_{T_1}]_{ii} A_{ij} [\mathbf{I}_{T_2}]_{jj} A_{jk} [\mathbf{I}_{T_3}]_{kk} A_{ki}A_{ij'} [\mathbf{I}_{T_2}]_{j'j'} A_{j'k'} [\mathbf{I}_{T_3}]_{k'k'} A_{k'i}\right]
\\ =&
p_1\ex\left[\sum_{j,k, j',k'=1}^N  A_{ij} [\mathbf{I}_{T_2}]_{jj} A_{jk} [\mathbf{I}_{T_3}]_{kk} A_{ki}A_{ij'} [\mathbf{I}_{T_2}]_{j'j'} A_{j'k'} [\mathbf{I}_{T_3})_{k'k'} A_{k'i}\right]
\\= &
p_1(p_2^2 ([\mathbf{A}^3]_{ii})2-2p_1(p2-p1)[\mathbf{A}\odot \mathbf{A}^2]_{ii}+(p_2-p_1)^2(\mathbf{A}\odot \mathbf{A})^3]_{ii}.
\end{align*}
Similar arguments worked for the off diagonal term. More specifically, for the offdiagonal terms, we need to estimate, 
\begin{align*}
&\ex\left[\exp\left(\g\sum_{i=1}^N \sum_{j,k=1}^N (I_{T_1})_{ii} A_{ij} (I_{T_2})_{jj} A_{jk} (I_{T_3})_{kk} A_{k\ell} [\mathbf{x}]_i[\mathbf{x}]_\ell \right)\right]
\end{align*}
Apply the same argument, we have, 
\begin{align}
&\pr\left[\bigg|
\sum_{i=1}^N \sum_{j,k=1}^N [\mathbf{I}_{T_1}]_{ii} A_{ij} [\mathbf{I}_{T_2}]_{jj} A_{jk} [\mathbf{I}_{T_3}]_{kk} A_{k\ell} [\mathbf{x}]_i[\mathbf{x}]_\ell \bigg|>t\right]\nonumber \\
\le & \exp\left(-\frac{t^2}{ 64\sum_{i,j=1, i\neq j}^N\ex[\Lam_i \Lam_j]+4{\tilde \Omega}^*} \right),\label{eqn:off_diag_inq}
\end{align}
with 
\begin{align*}
{\tilde \Omega}^*:= \prod_{i=1}^N\prod_{i=1}^N(1+p_1 \g [\mathbf{A}^3]_{ij} + 16p_1\g^2 [\mathbf{A}^6]_{ij} )-\left(1+ \g \sum_{i=1}^N\sum_{j=1}^N \ex[\Lam_i\Lam_j]\right),
\end{align*}
and  ${\tilde \Omega}^*$ is defined as $\inf\{M:{\tilde \Omega}\le \g^2 M, \g<1/(4eN^3)\}$, which is finite. 
Inequalities~\eqref{eqn:diag_inq} and~\eqref{eqn:off_diag_inq} allow us to estimate $L_n$. More precisely, given $(\epsilon, \delta)$, following the same arguments in~\cite{pmlr-v238-kalantzis24a}, 
\begin{align*}
\pr\left[ |L_n-\Delta(\mathbf{A})|> \frac{t}{np_1^3}\right] & \le \exp\left(-\frac{t^2}{ 64\g^2\sum_{i=1}^N\ex[\Lam^2_i]+4\Omega^*} \right)\\ &+\exp\left(-\frac{t^2}{ 64\g^2\sum_{i,j=1, i\neq j}^N\ex[\Lam_i \Lam_j]+4{\tilde \Omega}^*} \right).
\end{align*}
Therefore, we need to pick $t=p_1^3\epsilon$, and $n$ such that, 
\begin{align*}
\exp\left(-\frac{n^2p_1^6\epsilon^2}{ 64\g^2\sum_{i=1}^N\ex[\Lam^2_i]+4\Omega^*} \right) \le \delta/2,
\end{align*}
and 
\begin{align*}
\exp\left(-\frac{n^2p_1^6\epsilon^2}{ 64\g^2\sum_{i,j=1, i\neq j}^N\ex[\Lam_i \Lam_j]+4{\tilde \Omega}^*} \right)\le \delta/2
\end{align*}
Thus, $n_0=Cp_1^6\epsilon^{-2} \log(\delta^{-1})$ for a fixed $C>0$ suffices.
\end{proof}

\section{Numerical experiments} \label{sec4}

All experiments were conducted in \textsc{Matlab}~R2025a, using 64-bit IEEE 754 arithmetic. 
We apply partial Hutchinson's trace estimator with Rademacher variables on three graphs and 
vary the value of the ratio $f=\mathbb{E}[T]/N$ as $f\in \{0.2,0.6,1.0\}$. The option $f=1.0$ corresponds to classical Hutchinson's trace estimator \cite{Avron2010CountingTriangles}. 
The three test graphs chosen for our experiments are listed as: $a$) \emph{Erdos971}, which represents an Erdos collaboration 
network with $N=472$ nodes and $2628$ edges \cite{davis2011university}, $b$) a random binary, symmetric graph of size $N=5000$ and average degree equal to 15, and $c$) a symmetrized and binarized version of \emph{Harvard500}, which represents a 
network of $N=500$ webpages associated with the Harvard University \cite{moler2004numerical}. 
Each experiment was performed ten separate times, each with a different (fixed) seed to 
ensure reproducibility across sampling fractions and Monte Carlo realizations. 

The performance of all estimators is assessed using two complementary statistical 
quantities: the \emph{relative error} and the \emph{95\% confidence interval (CI)} of the 
estimated triangle count. In particular, let $\widehat{\Delta}_{f,n}({\cal G})\in \mathbb{R}$ 
denote the approximation of the total number of triangles $\Delta({\cal G})\in \mathbb{N}$ of 
the graph ${\cal G}$ after $n$ samples of Hutchinson's trace estimator, where each 
sample utilizes a row fraction $f$. The relative error of the approximation is defined as 
$\frac{\big| \widehat{\Delta}_{f,n}({\cal G}) - \Delta({\cal G}) \big|}{\Delta({\cal G})}$.  
On the other hand, the \emph{95\% confidence interval (CI)} serves as a statistical 
measure of the reliability of an estimator \cite{giles2015multilevel,robert2004monte}. 
For each sampling fraction $f$, the estimator produces an unbiased sequence of stochastic realizations of $\widehat{\Delta}({\cal G})$, from which we compute the empirical variance 
$s_f^2$ of the per-sample estimates. Under standard assumptions of independence and 
finite second moments, the Central Limit Theorem implies asymptotic normality of the 
sample mean, allowing us to approximate the confidence region as
\[
\widehat{\Delta}_{f,n}({\cal G}) \,\pm\, 1.96\,\dfrac{s_f}{\sqrt{n}}.
\] 
The resulting interval quantifies the statistical precision of the estimator, i.e., 
narrower intervals indicate that additional sampling yields diminishing returns, 
while wider intervals reflect greater variability due to aggressive masking or 
insufficient sampling. 


\begin{figure}[htbp]
    \centering
    \begin{subfigure}[b]{0.27\textwidth}
        \includegraphics[width=\textwidth]{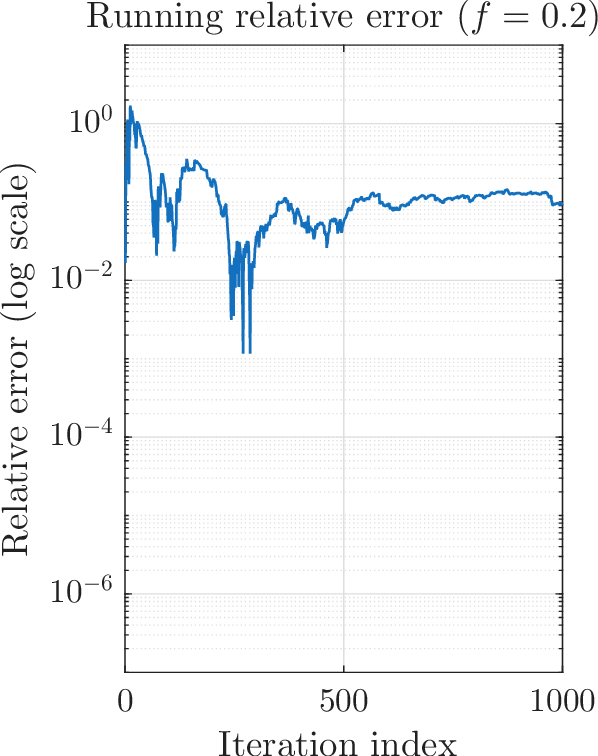}
    \end{subfigure}
    \begin{subfigure}[b]{0.27\textwidth}
        \includegraphics[width=\textwidth]{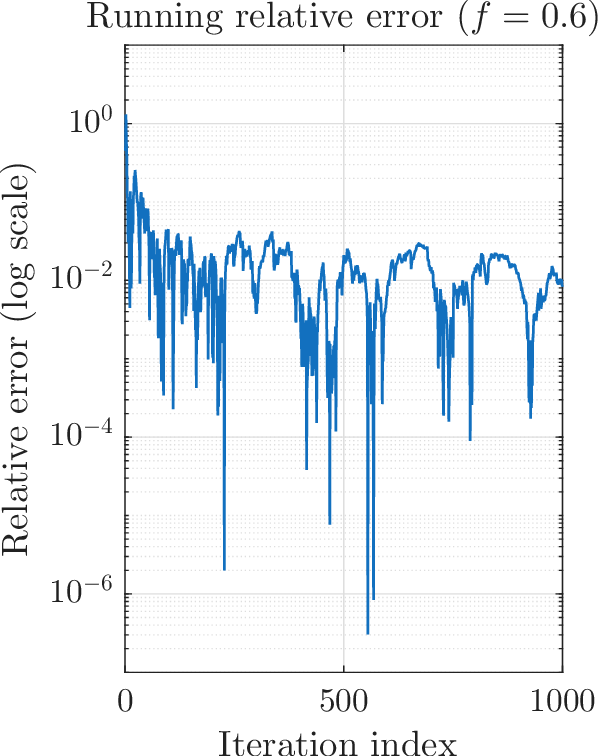}
    \end{subfigure}
    \begin{subfigure}[b]{0.27\textwidth}
        \includegraphics[width=\textwidth]{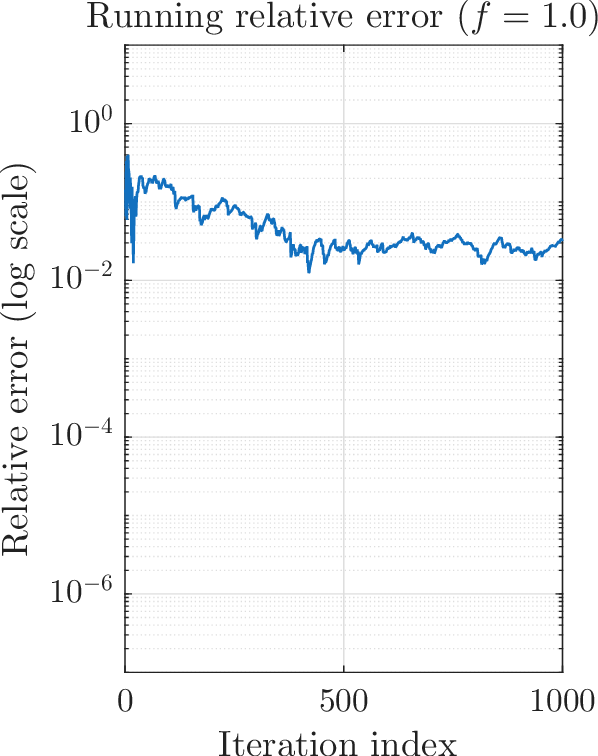}
    \end{subfigure}
    \begin{subfigure}[b]{0.27\textwidth}
        \includegraphics[width=\textwidth]{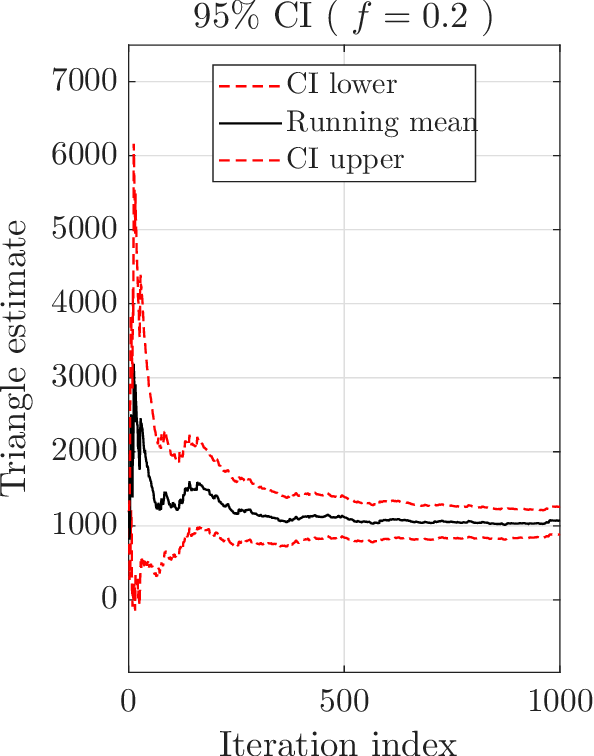}
    \end{subfigure}
    \begin{subfigure}[b]{0.27\textwidth}
        \includegraphics[width=\textwidth]{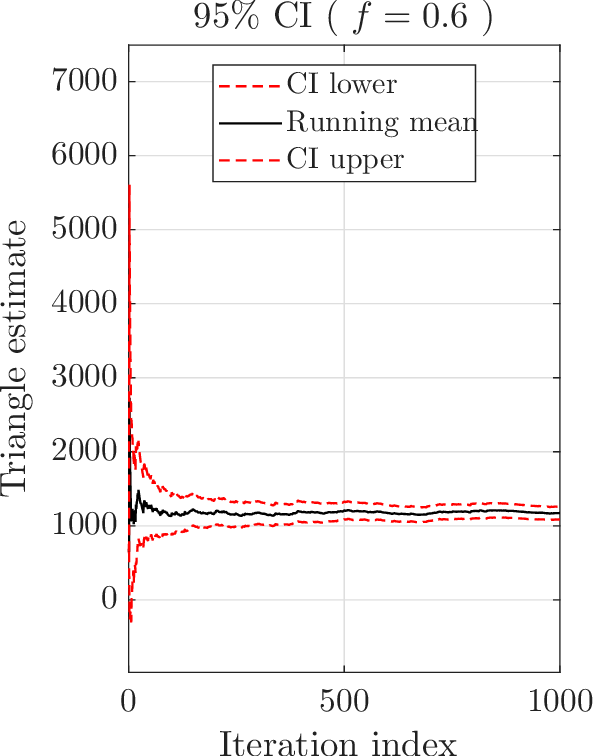}
    \end{subfigure}
    \begin{subfigure}[b]{0.27\textwidth}
        \includegraphics[width=\textwidth]{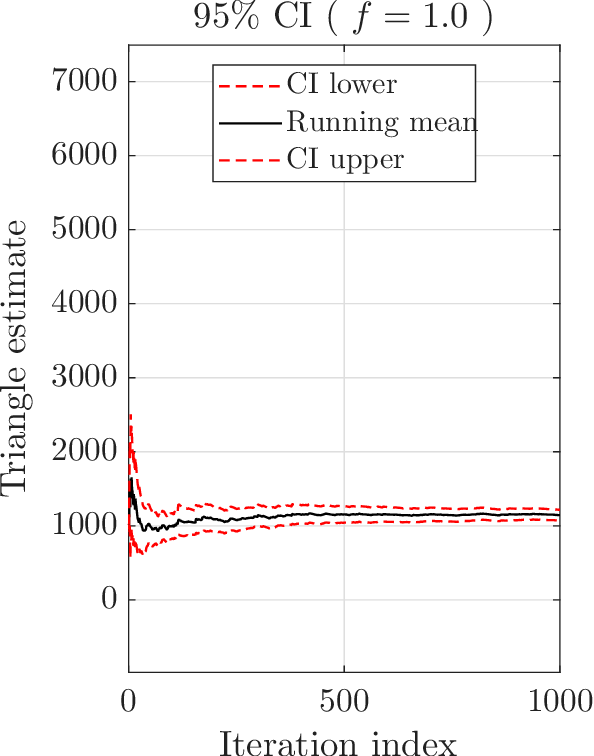}
    \end{subfigure}
    \caption{Cumulative relative error (top row) and 95\% CI (bottom row) for the graph \emph{Erdos971}.}
    \label{fig:1}
\end{figure}

\begin{figure}[htbp]
    \centering
    \begin{subfigure}[b]{0.27\textwidth}
        \includegraphics[width=\textwidth]{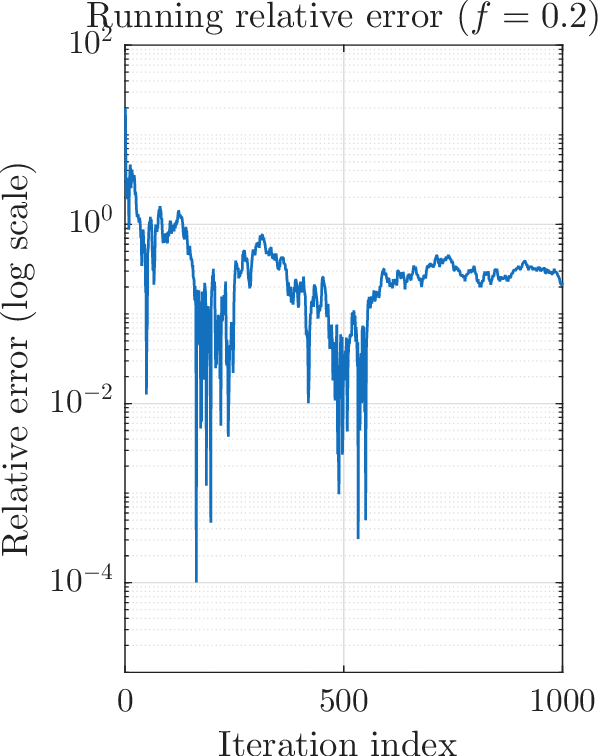}
    \end{subfigure}
    \begin{subfigure}[b]{0.27\textwidth}
        \includegraphics[width=\textwidth]{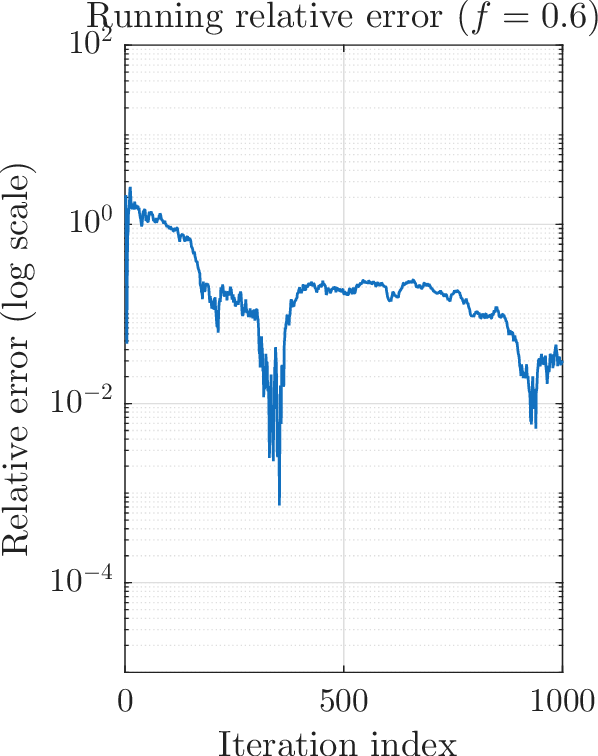}
    \end{subfigure}
    \begin{subfigure}[b]{0.27\textwidth}
        \includegraphics[width=\textwidth]{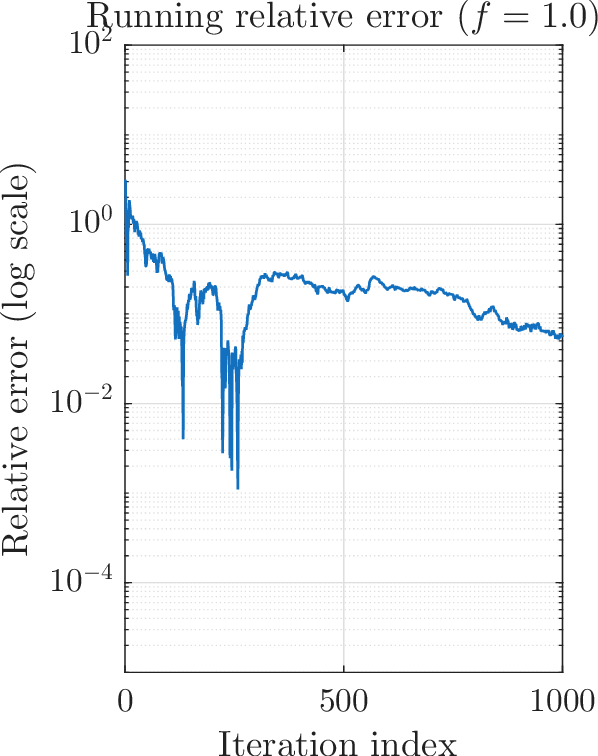}
    \end{subfigure}
    \begin{subfigure}[b]{0.27\textwidth}
        \includegraphics[width=\textwidth]{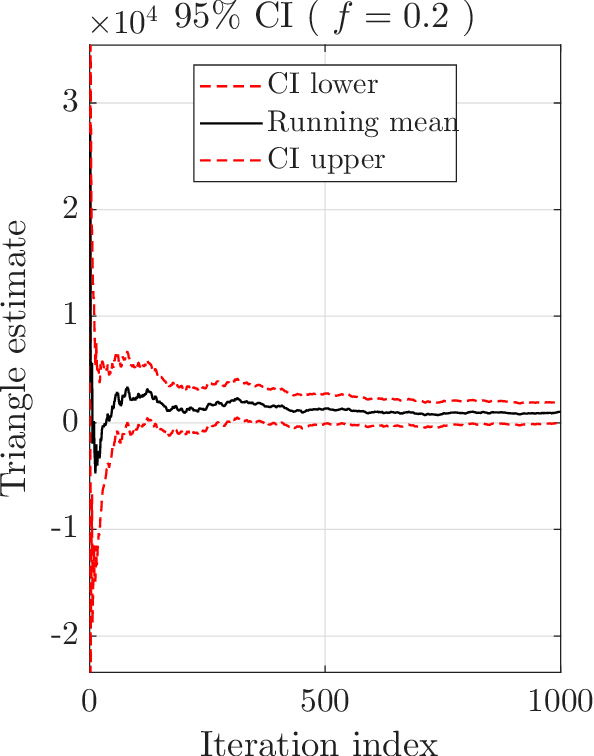}
    \end{subfigure}
    \begin{subfigure}[b]{0.27\textwidth}
        \includegraphics[width=\textwidth]{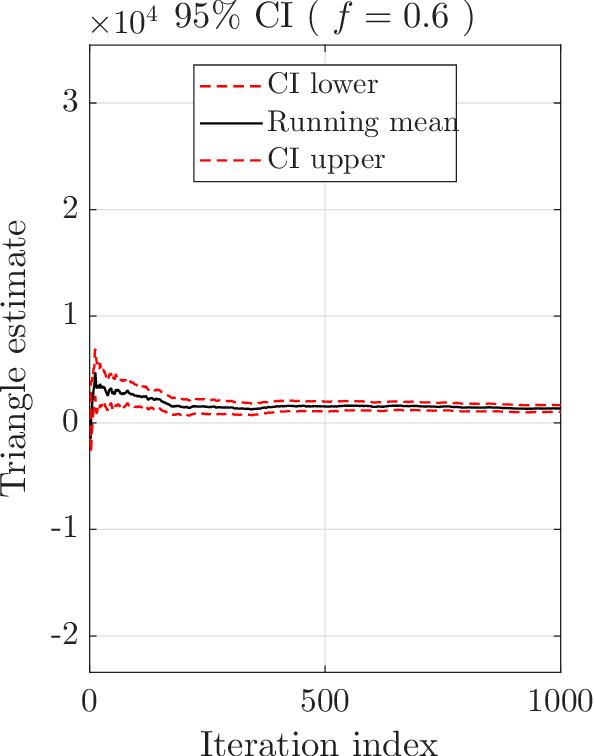}
    \end{subfigure}
    \begin{subfigure}[b]{0.27\textwidth}
        \includegraphics[width=\textwidth]{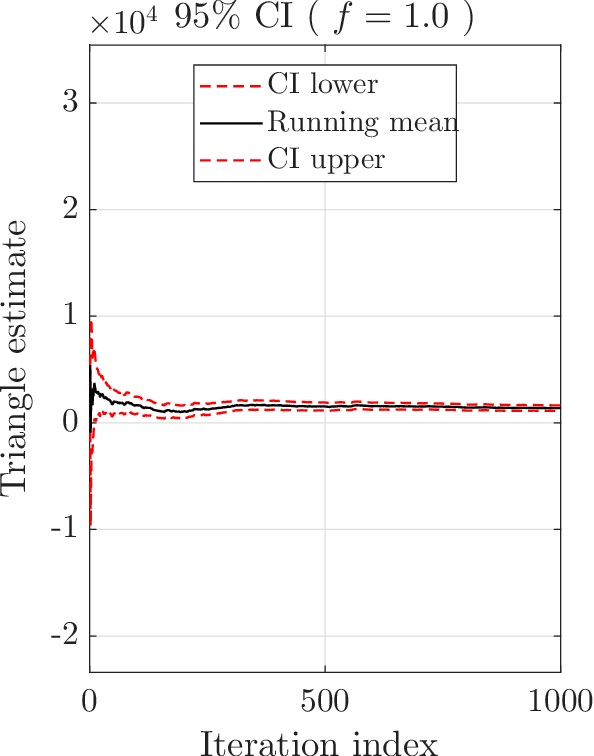}
    \end{subfigure}
    \caption{Cumulative relative error (top row) and 95\% CI (bottom row) for the random graph.}
    \label{fig:2}
\end{figure}

\begin{figure}[htbp]
    \centering
    \begin{subfigure}[b]{0.27\textwidth}
        \includegraphics[width=\textwidth]{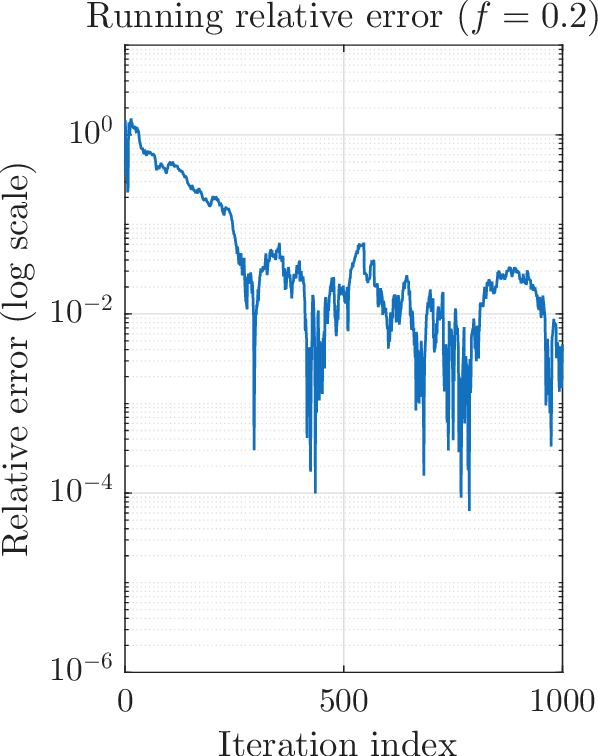}
    \end{subfigure}
    \begin{subfigure}[b]{0.27\textwidth}
        \includegraphics[width=\textwidth]{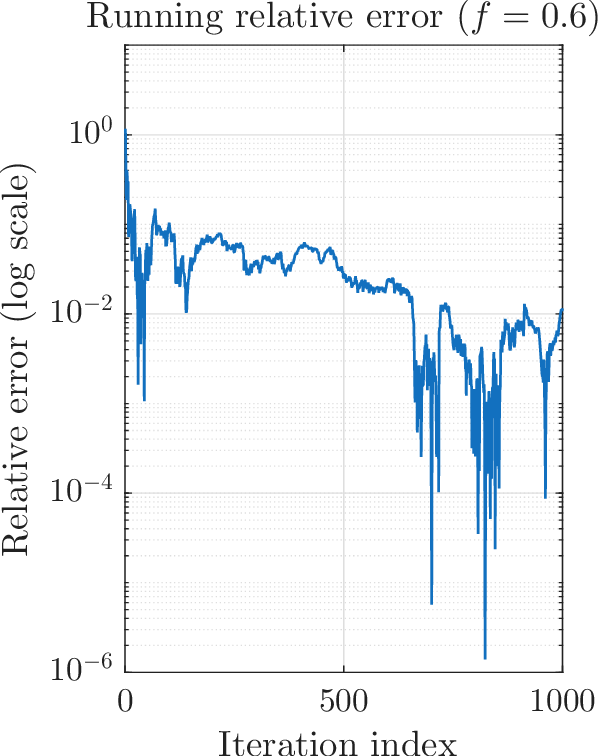}
    \end{subfigure}
    \begin{subfigure}[b]{0.27\textwidth}
        \includegraphics[width=\textwidth]{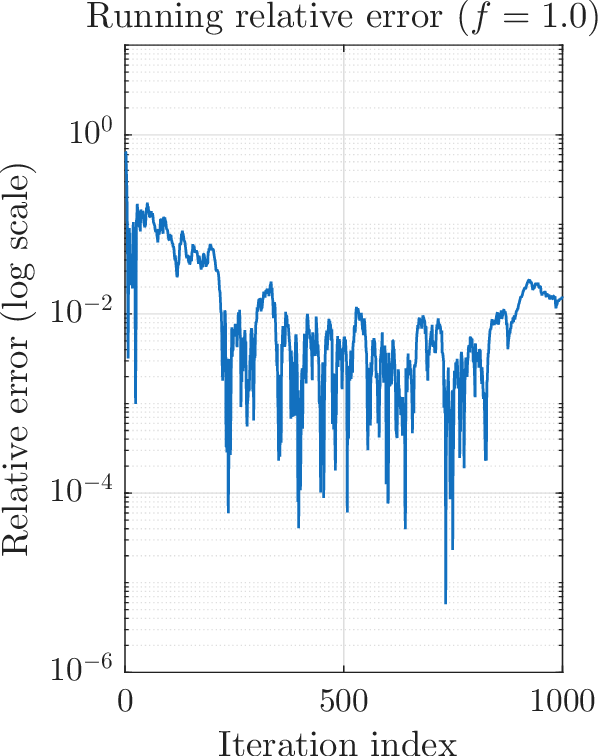}
    \end{subfigure}
    \begin{subfigure}[b]{0.27\textwidth}
        \includegraphics[width=\textwidth]{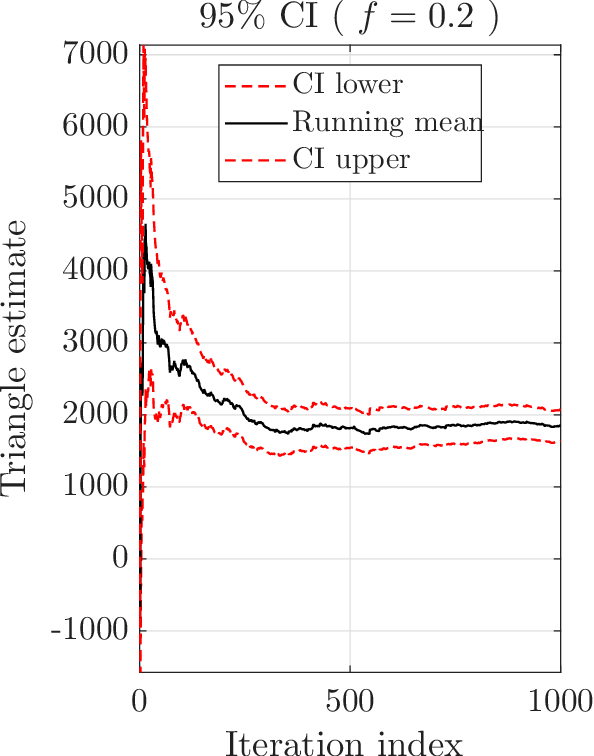}
    \end{subfigure}
    \begin{subfigure}[b]{0.27\textwidth}
        \includegraphics[width=\textwidth]{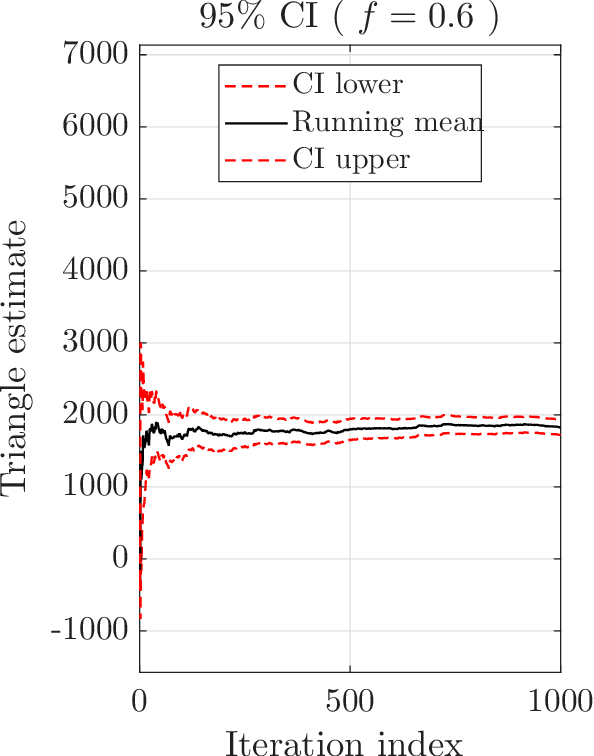}
    \end{subfigure}
    \begin{subfigure}[b]{0.27\textwidth}
        \includegraphics[width=\textwidth]{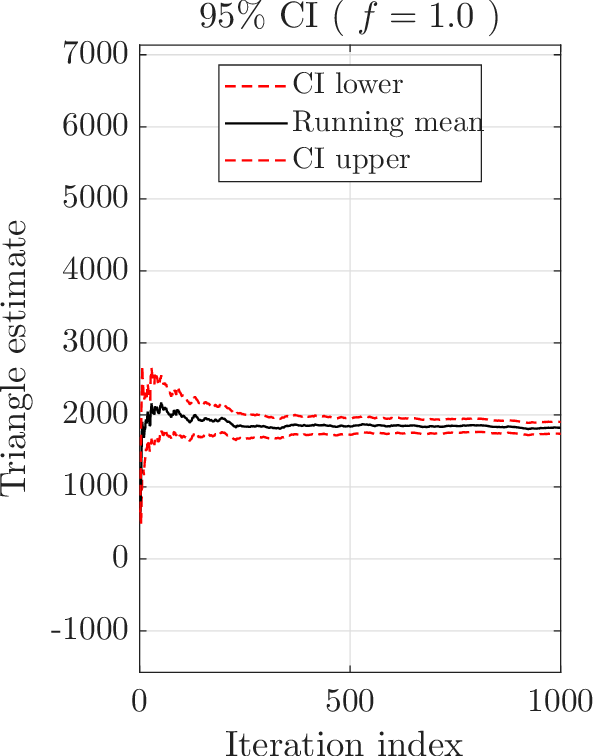}
    \end{subfigure}
    \caption{Cumulative relative error (top row) and 95\% CI (bottom row) for the graph \emph{Harvard500}.}
    \label{fig:3}
\end{figure}

Figures \ref{fig:1}-\ref{fig:3} plot the cumulative relative error of the triangle counting estimators as $n=1,2,\ldots,1000$, increases and $f$ varies, as well as the 95\% CI. As $n$, increases, the 
asymptotic behavior of each estimator improves to a more accurate approximation, although 
it is likely that better statistical approximations might occur for intermidiate values of 
$n$. Similarly, smaller values of $f$, generally lead to higher estimator variance and correspondingly wider 95\% confidence intervals. On the other hand, as $f$ increases, the confidence bounds become narrower, reflecting the anticipated reduction in sampling noise. 

\begin{figure}[htbp]
    \centering
    \begin{subfigure}[b]{0.27\textwidth}
        \includegraphics[width=\textwidth]{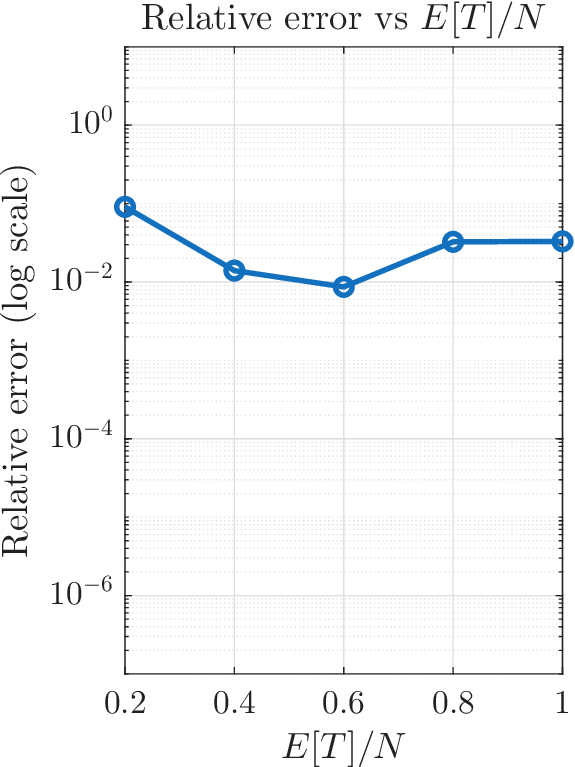}
    \end{subfigure}
    \begin{subfigure}[b]{0.27\textwidth}
        \includegraphics[width=\textwidth]{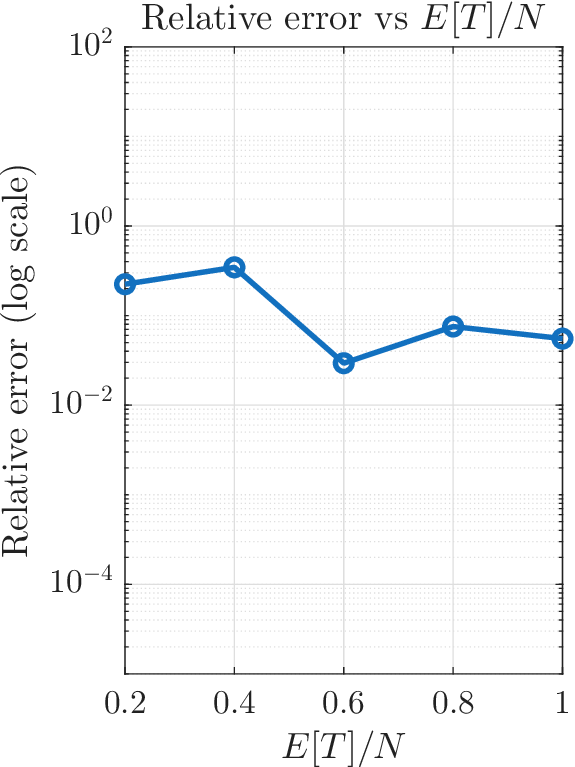}
    \end{subfigure}
    \begin{subfigure}[b]{0.27\textwidth}
        \includegraphics[width=\textwidth]{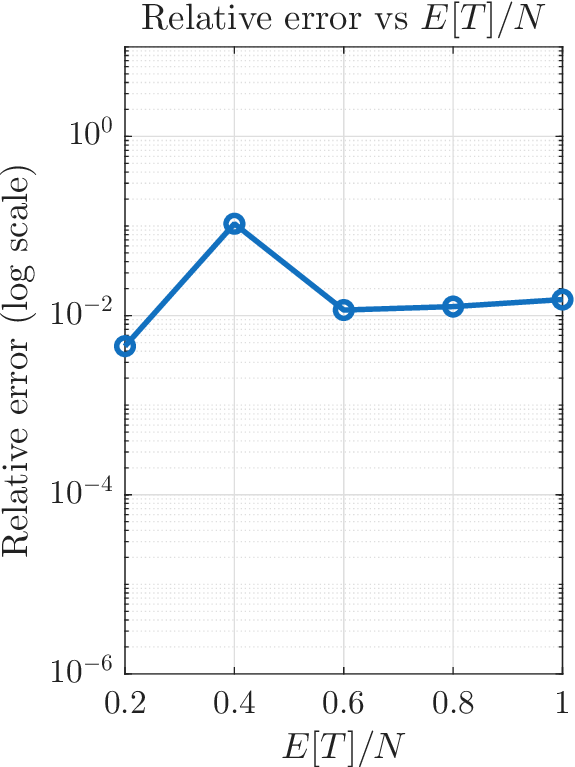}
    \end{subfigure}
    \begin{subfigure}[b]{0.27\textwidth}
        \includegraphics[width=\textwidth]{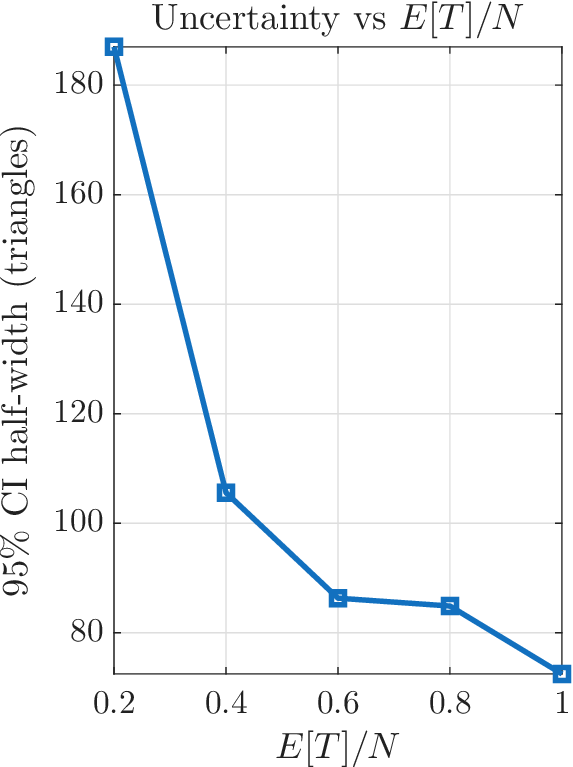}
    \end{subfigure}
    \begin{subfigure}[b]{0.27\textwidth}
        \includegraphics[width=\textwidth]{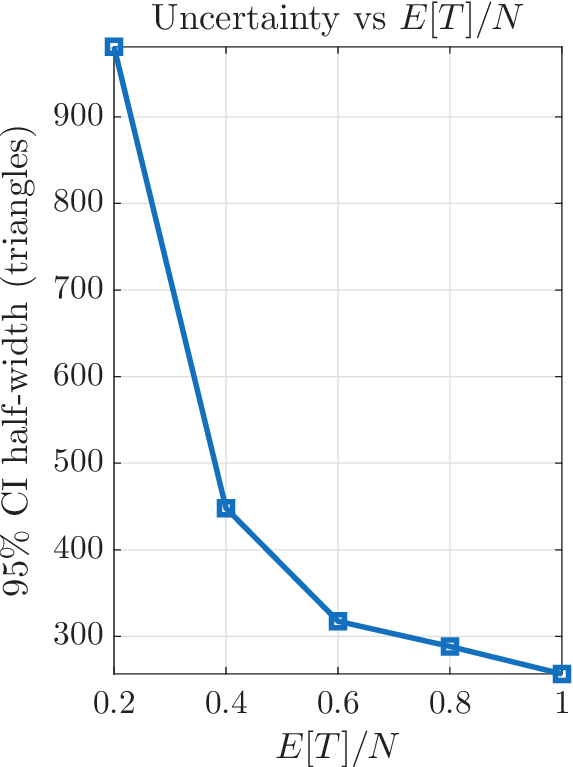}
    \end{subfigure}
    \begin{subfigure}[b]{0.27\textwidth}
        \includegraphics[width=\textwidth]{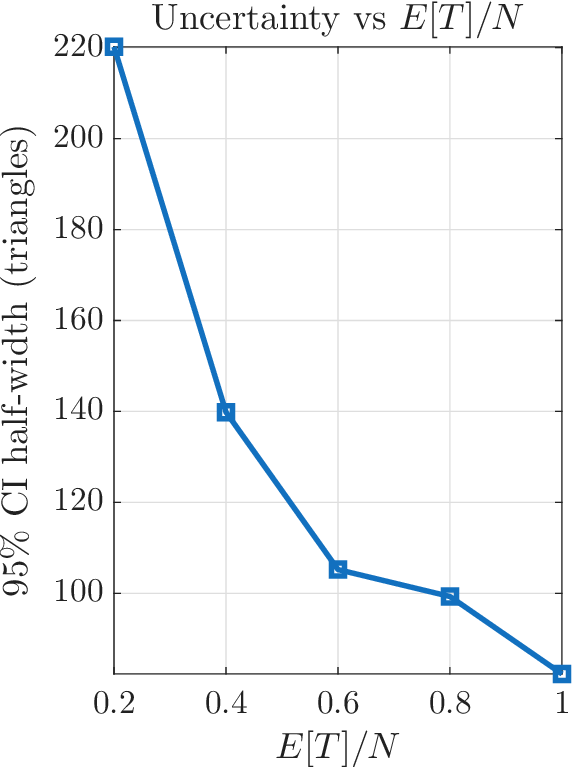}
    \end{subfigure}
    \caption{Top row: relative error as $\mathbb{E}[T]/N$ varies. Bottom row: 95\% CI half-width. From left to right: Erdos971, random graph with $N=5000$ nodes, and Hardvard500.}
    \label{fig:4}
\end{figure}

\section{Conclusion} \label{sec5}

In this paper, we considered the problem of estimating the number of triangles in large-scale graphs under constraints that limit full observation of matrix–vector products with the adjacency matrix. By extending Hutchinson’s randomized trace estimator, we introduced a novel variant tailored for scenarios where only partial and randomly selected entries of each product are available. Our theoretical analysis established unbiasedness, second-moment bounds, and sample-complexity guarantees for the proposed estimator, demonstrating its robustness even under asynchronous and incomplete observation regimes. Experimental results on synthetic and real-world graphs confirmed that the method achieves competitive accuracy while significantly reducing communication and synchronization costs in distributed environments. These findings highlight the potential of randomized algorithms to adapt to modern computational architectures, where data movement and heterogeneity increasingly dominate performance considerations. 

will focus on several directions. First, we aim to explore adaptive sampling strategies that dynamically adjust the number and location of observed entries based on variance estimates or graph structure, potentially improving efficiency and accuracy. Second, we plan to investigate extensions to streaming and dynamic graphs, where triangle counts must be updated incrementally as the graph evolves. Finally, a promising direction is to combine our partial-observation estimator with Hutch++, leveraging its ability to capture dominant spectral components before applying stochastic estimation to the residual subspace. Such integration could significantly reduce variance and improve convergence, especially for graphs with strong spectral decay. 

\bibliographystyle{plain} 

\end{document}